\documentclass[12pt]{msml2020} % Include author names

\usepackage{times}
\usepackage{graphicx}
\usepackage{caption}
\usepackage{amsfonts}
\usepackage{amsmath}
\usepackage{amssymb}
\usepackage{fancyhdr}
\usepackage{titlesec}
\usepackage{indentfirst}
\usepackage{booktabs}
\usepackage{verbatim}
\usepackage{color}
\usepackage{tcolorbox}
\usepackage{mathtools}
\usepackage{bm}
\usepackage{wasysym}
\usepackage{empheq}
\graphicspath{{figures/}} 
\usepackage[font=small,labelfont=bf]{caption} 
\usepackage{chngcntr}
\newcommand{\rom}[1]{\uppercase\expandafter{\romannumeral #1\relax}}
\usepackage[page,header]{appendix}
\usepackage{titletoc}

\numberwithin{equation}{section}
\newtheorem{assumption}[theorem]{Assumption}
\newtheorem{abc}{Algorithm}

\newcommand{\tcb}[1]{\textcolor{black}{#1}}

\newcommand{\argmin}{\operatornamewithlimits{argmin}}

\def\a{\alpha}

\def\b{\beta}
\def\d{\delta}
\def\e{\tilde{J}}
\def\g{\gamma}
\def\lam{\lambda}
\def\o{\omega}
\def\p{\phi}
\def\s{\sigma}
\def\th{\theta}

\def\th{\theta}

\def\O{\Omega}

\def\p{\rho}
\def\Up{\tilde{\Sigma}}

\def\E{\mathbb E}

\def\P{\mathbb P}
\def\R{\mathbb R}
\def\T{\mathbb T}
\def\S{\mathbb S}

\def\l{\left}
\def\r{\right}

\def\ll{\left\lVert}
\def\rl{\right\rVert}
\def\lv{\left\lvert}
\def\rv{\right\rvert}
\def\({\left(}
\def\){\right)}

\def\pt{\partial}
\def\nb{\nabla}

\def\ds{\displaystyle}

\def\qd{\quad}

\def\h{\hat}
\def\t{\tilde}

\def\stp{s_{m+1}}
\def\stpt{s_{m+2}}
\def\st{s_m}
\def\dt{\epsilon}
\def\sdt{\sqrt{\dt}}
\def\J{J}
\def\hJ{\hat{\J}}
\def\f{f}
\def\thst{\th^*}
\def\tp{{t+1}}

\def\dst{\Delta \st}
\def\dstp{\Delta \stp}
\def\hth{\h{\th}}
\def\Sig{\Sigma}
\def\up{\tilde{j}}

\def\p{p}
\def\pinf{p^\infty}
\def\hp{\h{p}}
\def\hpinf{\h{p}^\infty}
\def\hZ{\h{Z}}
\def\hb{\h{\b}}
\def\stt{*}
\def\hst{{\h{*}}}
\def\vst{V^*}
\def\t{m}

\DeclareMathOperator*{\Var}{\mathbb{V}}
\DeclareMathOperator{\diag}{diag}

\title[Borrowing From the Future]{Borrowing From the Future: An Attempt to Address Double Sampling}

\msmlauthor{%
  \Name{Yuhua Zhu} \Email{yuhuazhu@stanford.edu}\\
  \addr Department of Mathematics, Stanford University, Stanford, CA 94305
  \AND
  \Name{Lexing Ying} \Email{lexing@stanford.edu}\\
  \addr Department of Mathematics and ICME, Stanford University, Stanford, CA 94305
}

\begin{document}

\maketitle

\begin{abstract}%
For model-free reinforcement learning, one of the main challenges of stochastic Bellman residual minimization is the double sampling problem, i.e., while only one single sample for the next state is available in the model-free setting, two independent samples for the next state are required in order to perform unbiased stochastic gradient descent. We propose new algorithms for addressing this problem based on the idea of borrowing extra randomness from the future. When the transition kernel
varies slowly with respect to the state, it is shown that the training trajectory of new algorithms is close to the one of unbiased stochastic gradient descent. Numerical results for policy evaluation in both tabular and neural network settings are provided to confirm the theoretical findings.
\end{abstract}


\section{Introduction}
%Machine learning and artificial intelligence have achieved tremendous successes in recent years
%in many learning tasks. Among them,

Reinforcement learning (RL) has received a lot of attention following the recent successes, such as
AlphaGo and AlphaZero \citep{silver2016mastering,silver2017mastering}. At the center of RL is the
problem of optimizing Markov decision process (MDP), i.e., finding the optimal policy that maximizes
the return \citep{sutton2018reinforcement}. As a type of learning with minimal or no supervision, RL
can be more powerful than supervised learning and is often considered to be closer to the natural
learning process.  On the other hand, as an optimization and control problem, RL is also
significantly harder with many practical challenges, such as long trajectories required for
convergence, high dimensional continuous state and action spaces, learning with limited and noisy
samples, etc.

\paragraph{Background.}
This paper considers policy evaluation, also known as prediction, which is one of the most basic
problems of RL. In model-based RL, especially with small to medium-sized state space, value
iteration is commonly used in practice as it guarantees the convergence
\citep{bertsekas1996neuro}. In model-free RL, the temporal difference (TD) algorithm
\citep{sutton1988learning} converges in the tabular setting as well as with linear
approximation. However, the stability and convergence of TD are not guaranteed when the nonlinear
approximation is used \citep{boyan1995generalization,baird1995residual,tsitsiklis1997analysis}. With
the recent development and empirical successes of deep neural networks (DNNs), it becomes even more
important to understand and stabilize nonlinear approximations.

One direction for stabilizing the nonlinear approximation is to formulate the policy evaluation as a
minimization problem rather than a fixed-point iteration; one such example is the Bellman residual
minimization (BRM) \citep{baird1995residual}, which is sometimes also called Bellman error minimization in
the literature. However, BRM suffers from the so-called double sampling problem, i.e., at a given
state, two independent samples for the next state are required to perform unbiased stochastic
gradient descent (SGD). Such a requirement is often hard to fulfill in a model-free setting,
especially for continuous state space.

%\color{blue}
\paragraph{Related work.}
One way to avoid the double sampling problem in BRM is to consider the primal-dual
representation. This method turns the task into finding a saddle point of a minimax problem. Such a
method includes GTD \citep{sutton2008convergent,sutton2009fast} and its variants
\citep{bhatnagar2009convergent,maei2010toward,dai2017sbeed,liu2015finite}. However, the minimax
problem can be less stable than the original minimization problem when the maximum is taken over a
non-concave function. Therefore, a direct application of the primal-dual method to RL problem with
nonlinear function approximations and continuous state space might result in suboptimal performance
(see Section \ref{sec: compare} and Figure \ref{fig: eg1_compare} for details.)

Another approach for BRM chooses to solve the minimization problem directly.  The stochastic compositional gradient method (SCGD), proposed in \citep{wang2017stochastic,wang2016accelerating}, is a
two time-scale algorithm that addresses the double sampling problem in minimizing a function in the
form of two nested expectations. 
%In the mean-squared form of BRM, the basic SCGD method gives the same algorithm as the primal-dual method when linear approximation or tabular form is used {\bf LY: add reference for this}. 
%\YZ{I went through all the paper citing \cite{wang2016accelerating}, none of them mentioned that
%  their algorithm is the same as the primal-dual algorithm. So I deleted this sentence in the introduction, although I think it is correct. I guess maybe there is no paper actually using  primal dual to solve the bellman residual problem directly. GTD's objective function is not exactly $\E \d^2$, but $\E(\d\phi)^2$. Other PD paper, they either add regularizer in the objective function or using policy gradient to do RL. } \LY{if you are sure this is true, you should still point this out.}
When dealing with the continuous state space, the performance of the SCGD method is
less clear because the minimizer SCGD found is not necessarily the fixed point of the Bellman
operator (see Section \ref{sec: compare} for details.)
%However, it requires careful choices of two different step sizes, and usually one step size is smaller than the other, it could slow the convergence rate.
\color{black}

\paragraph{Contributions.}
In this paper, we revisit the Bellman residual minimization and develop two algorithms to alleviate
the double sampling problem.  The key idea of the new algorithms is to {\em borrow extra randomness
  from the future}. When the transition kernel varies slowly with respect to the state, we show that
the training trajectories of the proposed algorithms are statistically close to the one of the
unbiased SGD. The proposed algorithms are applied to the prediction problem (i.e., policy
evaluation) in both the tabular and neural network settings to confirm the theoretical findings.
\tcb{We also show that, for continuous state space, our method results in better performance than
  the primal-dual method and SCGD.} Though the discussion here focuses on policy evaluation, the
same techniques can be extended to Q-Learning \citep{watkins1989learning} or value iteration.
%\tcb{Our algorithm could also be generalized to general minimization problems of a composition of two expected-value functions. }

\paragraph{Organization.}
The rest of this paper is organized as follows. \tcb{Section \ref{sec: model} introduces the key
  idea of the proposed algorithms and a summary of the paper's main results.  A discussion of
  related work is given in Section \ref{sec: compare}. Section \ref{sec: algo} gives the details of
  the proposed algorithms.} Section \ref{sec: thm pf} bounds the errors between the new algorithm
and the accurate but unrealistic {\em uncorrelated sampling} algorithm. Numerical results are given
in Section \ref{sec: numerics} to confirm the theoretical findings and demonstrate the efficiency of
the new algorithms.

%% , and the resulting gradient is close to the unbiased gradient of the
%% BRM.  It
%% requires the environment to be homogeneous, which has many applications in the physical world. We
%% apply the new algorithm to the value evaluation by Neural network and Tabular representation in
%% Section \ref{numerics}. The performance is good in both cases.

%Our main contribution is the following: 
%\begin{itemize}
%\item We propose two efficient methods to find minimizer of the double sampling problem when evaluate the value function under a given policy.

%\item We show that the difference of the new algorithm with the unbiased SGD is $O(\frac{\dt^2}{\eta})$.

%\end{itemize} 

%======================================================
\section{Models and key ideas}

\label{sec: model}

%\subsection{Models}

This paper considers both continuous and discrete state space models.
\begin{itemize}
\item In the {\em continuous state space} setting, we consider a discrete-time Markov decision
  process (MDP) with continuous state space. The state space $\S\subset\R^{d_s}$ is a compact set.
  Given a fixed policy, we assume that the one-step transition $P(s,s')$ is prescribed by an unknown
  drift $\a(\cdot)$ and an unknown diffusion $\s(\cdot)$,
  \begin{equation}
    \label{eq: discrete St 0}
    \stp = \st+\a(\st)\dt + \s(\st) \sdt Z_m, \qd Z_m\sim \text{Normal}(0,I_{d_s\times d_s}).
  \end{equation}
When the state reaches the boundary, it follows a prescribed boundary condition. Here, we choose
  to work with a stochastic differential equation (SDE) setting to simplify the presentation
  of the algorithms and the theorems. The scalings $\dt$ and $\sdt$ of the drift and the noise terms
  correspond to discretizing the SDE with the time step $\dt$. However, both the algorithms and theorems can
  be extended to more general cases.  The real-valued immediate reward is denoted by $r(s,s')$ for
  $s,s'\in\S$, and the discount factor $\g$ is in $(0,1)$. 
\item In the {\em discrete state space} setting, $\S=\{0,1,\ldots,n-1\}$. Given a policy, the
  transition matrix $P(s,s')$ that represents the probability from state $s$ to $s'$
  \begin{equation*}
    P(s,s') = P(\stp = s'| \st = s)
  \end{equation*}
  is assumed to vary slowly in both $s$ and $s'$. The function $r(s, s')$ denotes the immediate
  reward function from state $s$ to $s'$, under the current policy. Finally, the discount rate $\g$
  is between $0$ and $1$.
\end{itemize}

Since we only consider the prediction problem here, the policy is considered fixed for the rest of
the paper. With these notations, the value function $V(s)$ is the expected discounted return if the
policy is followed from state $s$,
\begin{equation}
  \label{eq: discrete St}
  V(s) = \E\left[\sum_{m=0}^\infty \g^tr(s_m,s_{m+1})|s_0 = s\right].
\end{equation}
Let $R(s) = \E[ r(s_m,s_{m+1})|\st = s]$ be the immediate reward under the fixed policy, and $\T$ be
the Bellman operator defined as
\begin{equation}
  \label{eq: discrete St}
  \T V(s) = R(s)+\g\E[ V(\stp) |\st = s].
\end{equation}
The value function $V(s)$ is the fixed point of the operator $\T$.

\subsection{Main ideas and results}
Let us consider approximating the value function in a parameterized form $V(s,\th)$ in model-free
RL. Here, the value function approximation can either be linear or nonlinear with respect to the
parameter $\th$. One way for computing the optimal parameter $\thst$ is to perform gradient descent
to the so-called mean-square Bellman residual
\begin{equation}
  \label{eq: obj fcn RL}
  \min_\th\  \E\l(\E \l[R(\st) + \g V(\stp;\th) - V(\st; \th)|\st\r]\r)^2.
\end{equation}
This approach is thus called the Bellman residual minimization (BRM) in the literature. The
stochastic gradient descent of BRM is based on an unbiased gradient estimation of the objective
function \eqref{eq: obj fcn RL}, which requires two independent transitions from $\st$ to $\stp$.
However, for model-free RL, one does not know explicitly how the environment interacts with the
agent. That is to say, besides observing the agent's trajectory $\{\st\}_{m=0}^T$ under the
given policy, one cannot generate $\stp$ from $\st$ because one does not know the drift and
diffusion in (\ref{eq: discrete St 0}) explicitly. Since the trajectory $\{\st\}_{m=0}^T$ provides
only one simulation from $\st$ to $\stp$, there is no direct means to generate the second copy. This
is the so-called double sampling issue.

The main idea of this paper is to alleviate the double sampling issue by {\em borrowing randomness
  from the future}: instead of requiring a new copy of $\stp'$ starting from $\st$, we approximate
it using the difference between $\stpt$ and $\stp$
\[
\stp' \approx \st + (\stpt - \stp).
\]
When the derivatives of the drift and diffusion terms are under control, the difference between $\dst$
and $\dstp$ is small, which makes the new $\stp'$ statistically close to the distribution of the
true next state. We refer to the algorithm based on this idea the BFF algorithm, where BFF is
short for ``borrowing from the future''.

Before diving into the algorithmic details, let us first summarize the main properties of the BFF
algorithm. When the derivatives of the implicit drift and diffusion terms are small, we are able to
show, in the continuous state space setting,
\begin{itemize}
\item The difference between the biased objective function in the BFF algorithm and the true
  objective function is only $O(\dt^2)$ (Lemma \ref{lemma: eps est});
\item The equilibrium distribution of $\th$ of the BFF algorithm differs from the unbiased SGD
  within an error of order $O(\frac{\dt^2}{\eta})$ (Theorem \ref{thm: p_infty});
\item The evolution of $\th$ of the BFF algorithm differs statistically from the unbiased SGD only
  within an error of $O(1+\frac{\dt^2}{\eta})O(\dt^2)$ (Theorem \ref{thm: p_t}).
\end{itemize}
Here $\eta$ is the ratio of the learning rate over the batch size. Note that in order to have the
error under control, $\eta$ cannot be too small. Intuitively, this is because: the algorithm
minimizes a slightly biased objective function, so if the optimization is done without any
randomness, the solution will have little overlap with the exact one.

\tcb{Although the theoretical results only concern with the first case, we verify numerically in
  Section \ref{sec: numerics} that the proposed algorithms perform well also in the discrete state
  space setting. }

%\color{blue}
\subsection{Discussion on related work}
\label{sec: compare}
%\YZ{Previously, I wrote the objective function of SCGD as $(\E[\d])^2$, but I think this might raise  some comfusion. So I changed it in this version.}\LY{OK}

There are two other ways commonly used to solve the BRM problem. One is the SCGD method proposed in
\citep{wang2016accelerating}, the goal of which is to solve the minimization problem in the form of
\begin{equation*}
    \min_{\th} \E_v[f_v( \E_\o[g_\o(\th)])].
\end{equation*}
When $V(s;\th): \S \to \R^{|\S|}$ is in the vector form, one could set $g_{\stp} = R(\st) + \g V(\stp) - V(\st)$ as a vector function in $\R^{|\S|}$,  and $f = \sum_{i}g_i^2$.  Then the above minimization problem automatically becomes
\begin{equation*}
    \min_{\th} \E \lv \E[ R(\st) + \g V(\stp) - V(\st) |\st]  \rv^2,
\end{equation*}
which gives the fixed point of the Bellman operator $\T$. According to Algorithm 1 in \citep{wang2017stochastic}, the basic SCGD updates the parameter $\th$ in the following way,
\begin{equation}
  \label{algo: GTD}
  \begin{aligned}
    &y_{k+1} = y_k + \beta \l(R(\st) + \g V(\stp ;\th_k) - V(\st;\th_k) - y_k\r);\\ &\th_{k+1} =
    \th_k - \tau \l(\g\nb_\th V(\stp ;\th_k) -\nb_\th V(\st;\th_k) \r)y_{k+1}.
  \end{aligned}
\end{equation}
where $y\in\R^{|\S|}$ is a vector. 

When the state space $\S$ is continuous, $V(s;\th)\in \R$ is usually represented as a scalar function. Note that simply set $g = R(\st) + \gamma V(\stp) - V(\st)$ and $f = g^2$ does not necessarily give the fixed point of the Bellman operator because there is no conditional expectation of $\stp$ on $\st$ in the objective function. So it is less clear  how to apply SCGD to the case of the continuous state space.

%\citep{wang2016accelerating, wang2017stochastic}, though the objective function of SCGD shown as follows is slightly different from ours \begin{equation}\label{obj SCGD}\min_\th\  \lv\E\l[\T V(s;\th) - V(s; \th)\r]\rv^2.\end{equation} When $V(s;\th): \S \to \R^{|\S|}$ is in vector form, the above objective function becomes, \begin{equation*}\min_\th \qd \sum_{\st\in\S}\l(\T V(\st;\th) - V(\st; \th)\r)^2\rho^2(\st),\end{equation*} where $\rho(\st)$ is the stationary probability of the states.  Then it is similar to our objective function with a different measure, \begin{equation*}\min_\th \qd  \E[ \lv \T V(s;\th) - V(s; \th)\rv^2] = \sum_{\st\in\S}\l(\T V(\st;\th) - V(\st; \th)\r)^2\rho(\st).\end{equation*} However, when $V(s; \th):\S \to \R$ is represented as a scalar function, the minimization problem \eqref{obj SCGD} leads to $\E [\T V(s)] = \E [V(s)]$ instead of the fixed point of the Bellman operator. Therefore, direct application of SCGD cannot be extended to continuous state space or neural network function approximation settings (unless the output of the neural network is in the xvector form).

\color{black}
The second method is based on the primal-dual formulation. The GTD method proposed in
\citep{sutton2008convergent} is later shown in \citep{liu2015finite} to a primal-dual algorithm for
the mean-squared temporal difference minimization problem. The primal-dual representation of the
objective function \eqref{eq: obj fcn RL} is
\begin{equation*}
  \min_\th\max_{y\in\R^{|\S|}}  \E\l[ \l(\E\l[R(\st) + \g V(\stp;\th) - V(\st; \th)|\st\r]\r) y(\st) - \frac{1}{2}y(\st)^2 \r].
\end{equation*}
It is easy to check that for a fixed $\th$, the maximum over $y$ gives the same objective function
\eqref{eq: obj fcn RL}. The SGD method for the above minimax problem is \eqref{algo: GTD}, which is the same as SCGD.

However, when the state space is continuous, $y(\cdot)$ usually is represented by a nonlinear
function, then the minimax problem becomes
\begin{equation}
  \label{obj nn pd}
  \min_\th\max_{\o\in\R^{|\S|}}  \E\l[ \l(\E\l[R(\st) + \g V(\stp;\th) - V(\st; \th)|\st\r]\r) y(\st;\o) - \frac{1}{2}y(\st;\o)^2 \r],
\end{equation}
which makes the maximum problem non-concave. Solving the maximum problem over $\o$ by (stochastic)
gradient descent will not necessarily give the objective function \eqref{eq: obj fcn RL}, which
brings in more instability to the performance of the primal-dual algorithm. (See Figure \ref{fig:
  eg1_compare} in Section \ref{sec: cont}. )

To summarize, when compared with the other methods, BFF has an advantage for continuous state space
model-free RL problems. BFF also gives comparable results in discrete state space with the above
methods mentioned.  However, BFF requires the smoothness assumption of the underlying dynamics,
while such an assumption is not necessary for applying the primal-dual method or the SCGD.

\color{black}

%=============================
\section{Algorithms}
\label{sec: algo}

This section describes the BFF algorithms in both the nonlinear approximation setting and the
tabular setting.

\subsection{Algorithm with nonlinear approximation}
Let us write the BRM (\ref{eq: obj fcn RL}) in an abstract form,
\begin{equation}
  \label{eq: obj fcn}
  \thst = \min_{\th \in \O} \J(\theta),\quad
  %\J(\theta) := \E j(\st; \theta)  := \E \l[ \frac{1}{2}\l(\E\l[f(\st, \stp; \th)|\st\r]\r)^2\r].
  \J(\theta) := \E \l[ \frac{1}{2}\l(\E\l[f(\st, \stp; \th)|\st\r]\r)^2\r],
\end{equation}
where $\O\subset\R^d$ is a compact domain, and
\[
f(\st,\stp;\th) = R(\st) + \g V(\stp;\th) - V(\st;\th)
\]
is the Bellman residual (sometimes also called Bellman error in the literature). Note that, when $V$
is approximated by a neural network with standard activation functions, the boundedness of $\th$ and
$\st$ implies that $V$ is also bounded. Following \eqref{eq: discrete St 0}, define
\begin{equation}
\label{eq: def of ds}
\dst := \stp-\st =  a(\st)\dt+\s(\st)\sdt Z_m.
\end{equation}

%$\st\in\S\subset\R^{d_s}$ is a discrete Markov chain,
%% \begin{equation}
%% \label{eq: discrete St}
%%     \stp = \st+\a(\st)\dt + \s(\st) \sdt Z_m, \qd Z_m\sim \text{Normal}(0,I_d).
%% \end{equation}
%% or equivalently, 

Suppose now that functions $\a(s)$ and $\s(s)$ were known explicitly. SGD with uncorrelated samples
at each state updates the parameter $\th$ as follows
\begin{abc}  \label{algo: uncor}
  {\bf [Uncorrelated sampling] }Given a trajectory $\{s_m\}_{m=0}^T$, at step $k$, randomly
  select $M$ elements from $\{0, \cdots, T\}$ to form the index subset $B_k$, generate a new
  $\stp'$ from $\st$ according to \eqref{eq: discrete St 0}, and update
  \begin{equation*}
    \th_{k+1} = \th_k - \frac{\tau}{M}\sum_{m\in B_k}\f(\st, \stp; \th_k)\nb_\th \f(\st, \stp';
    \th_k),
  \end{equation*}
  where $\tau$ is the learning rate, and $M$ is the batch size.
\end{abc}
However, as we pointed out already, generating another $\stp'$ is impractical as $\a(s)$ and $\s(s)$
are unknown in the model-free setting. Instead, the following {\em Sample-cloning} algorithm is
sometimes used.
\begin{abc}  \label{algo: cor}
  {\bf [Sample-cloning] } Given a trajectory $\{s_m\}_{m=0}^T$, at step $k$ with $\tau, B_k, M$ the
  same as in Algorithm \ref{algo: uncor}, update
  \begin{equation*}
    \th_{k+1} = \th_k - \frac{\tau}{M}\sum_{m\in B_k}\f(\st, \stp; \th_k)\nb_\th \f(\st, \stp;
    \th_k).
  \end{equation*}
\end{abc}
Note that the gradient in the above algorithm is not an unbiased gradient of the objective function
in (\ref{eq: obj fcn}). In fact, it is an unbiased gradient of
$\E\l[\frac{1}{2}\E\l[f(\st,\stp;\th)^2|\st\r]\r]$. We shall see in Section \ref{sec: numerics} that
this algorithm fails to identify the true solution $\thst$ even if the underlying drift and
diffusion terms are smooth.

\paragraph{Borrow from the future.}
Below we propose two algorithms that approximate the minimizer efficiently when the underlying drift
term $\a(s)$ and diffusion term $\s(s)$ change smoothly. Instead of minimizing $J(\th)$, the first
algorithm minimizes $\hJ(\th)$,
\begin{equation}  \label{eq: hat obj fcn}
  \min_{\th \in \O} \hJ(\th),\quad
  \hJ(\th):= \frac{1}{2}\E \l[ \E\l[f(\st, \stp; \th)|\st\r]\E\l[f(\st, \st+\dstp; \th)|\st\r]\r]
\end{equation}
where $\dstp$ is defined as (\ref{eq: def of ds}). The main idea is to approximate $\stp'=\st+\dst$
in Algorithm \ref{algo: uncor} with $\stp' \approx \st + \dstp$, i.e., creating another simulation
of $\st \to \stp$ by borrowing randomness from the future step $\stp \to \stpt$. When $\dt$ is
small, and the change of the drift and the diffusion are small as well, we expect the approximation
should be close to the unbiased gradient. Due to the independence between $\dst$ and $\dstp$, we
have
\begin{equation}  \label{eq: hat obj fcn 2}
  \hJ(\th) = \frac{1}{2}\E\l[f(\st, \stp; \th)f(\st, \st+\dstp;\th)\r].
\end{equation}
From \eqref{eq: hat obj fcn 2}, one can directly apply SGD algorithm to update the parameter $\th$
from the observed trajectory $\{s_m\}_{m=0}^T$.

\begin{abc}  \label{algo: shiftonloss}
  {\bf [BFF-loss] }Given a trajectory $\{s_m\}_{m=0}^T$, at step $k$, $\tau, B_k, M$ are the same as
  in Algorithm \ref{algo: uncor}
  \begin{equation*}
    \begin{aligned}
      \th_{k+1} = \th_k -\frac{\tau}{M}\sum_{m\in B_k} \frac{1}{2}\nb_\th \l[\f(\st, \stp; \th_k) \f(\st, \st+\dstp; \th_k)\r] ,
    \end{aligned}
  \end{equation*}
  where $\dstp = \stpt - \stp$.
\end{abc}

An alternative algorithm is applying the same technique directly on the unbiased gradient of the
true objective function. We will show in Section \ref{sec: numerics} that the two new algorithms
behave similarly in practice.
\begin{abc}  \label{algo: shiftongd}
  {\bf [BFF-gradient] } Given a trajectory $\{s_m\}_{m=0}^T$, at step $k$, $\tau, B_k, M$ are the same as in Algorithm \ref{algo: uncor}
  \begin{equation*}
    \th_{k+1} = \th_k - \frac{\tau}{M}\sum_{m\in B_k} \f(\st, \stp; \th_k) \nb_\th\f(\st, \st+\dstp;
    \th_k),
  \end{equation*}
  where $\dstp = \stpt - \stp$.
\end{abc}

Although these new BFF algorithms are biased when compared with Algorithm \ref{algo: uncor}, Section
\ref{sec: thm pf} proves that the bias of Algorithm \ref{algo: shiftonloss} is small. More
specifically, we show that the differences between Algorithms \ref{algo: uncor} and \ref{algo:
  shiftonloss} are all under control, in terms of the objective function, the evolution of SGD, and
the steady-state distribution.

%The theoretical proof for Algorithm \ref{algo: shiftongd} will be left for future study.

\subsection{Algorithms in the tabular setting}
\label{sec: tabular}
For tabular form, the value function $\vst\in\R^{n}$
satisfies the following Bellman equation,
\begin{equation*}
  \vst = \T(\vst) = r +\g P \vst.
  %, \qd \text{where}\qd P_{i,i'} = P(s_i,s_{i'}).
\end{equation*}
In the discrete setting, the BRM becomes
\begin{equation*}
  \vst = \argmin_{v\in\R^{n}}\frac{1}{2} \ll r +\g P v - v\rl_\mu^2,
\end{equation*}
where $\mu$ is the stationary distribution of the Markov chain. The gradient of the above objective
function can be written as
\begin{equation}  \label{eg2: nb J 1}
  \nb_v J = (\g P - I)^\top \diag(\mu) ( r + \g P v - v).
\end{equation}
Since $P$ appears twice in the above formula, in order to obtain an unbiased approximation of the
above gradient, we need two simulations from $\st$ to $\stp$. Given a trajectory $\{\st\}_{m=1}^T$,
choose a state $\st = i$ and the next state $\stp = j$. Assuming that the Markov chain reaches
equilibrium, such a choice is an unbiased estimate for $\diag(\mu)$. If the transition matrix $P$
were known, we could generate a new state $\stp' = s_{j'}$ and construct an unbiased estimation of
$\nb_v J$: replacing the first and second copies of $P$ in \eqref{eg2: nb J 1} with $P_1$ and $P_2$
given as follows:
\[
(P_1)_{il} =
\begin{cases}
  1,  &l=j'\\
  0,  &\text{otherwise}
\end{cases},
\quad
(P_2)_{il} =
\begin{cases}
  1, &l=j\\
  0, &\text{otherwise}
\end{cases}.
\]
Equivalently, the unbiased estimate of the gradient can be written as
\begin{equation}
  \label{eg2: unbiased nb J}
  (\nb_v J)_i = -(r_i + \g v_j-v_i),\quad
  (\nb_v J)_{j'} = \g(r_i + \g v_j-v_i),\quad
  (\nb_v J)_l = 0, \qd \forall l \neq i, j'.
\end{equation}

However, in the setup of model-free RL, without knowing the transition matrix $P$, one needs to
approximate $\vst$ based on the observed trajectory $\{\st\}_{m=1}^{T}$ alone. The BFF idea $\stp'
\approx \st+(\stpt-\stp)$ can also be applied here to give rise to the two BFF algorithms in the
tabular form. Below are the pseudocodes for Algorithms \ref{algo: uncor}-\ref{algo: shiftongd} in
the tabular form, where $v$ is updated based on an estimate $G$ of the true gradient $\nb_v J$,
\begin{equation*}
  v_{k+1} = v_k - \eta G_{m_k},
\end{equation*}
where $m_k$ is randomly selected from $\{1, \cdots, T\}$ and $G_m$ (dropping the $k$ index for
notation convenience) is computed differently as follows in the four algorithms.
\begin{itemize} 
\item {\bf Uncorrelated sampling:} Assume $\st = i, \stp = j$, (same for the other three
  algorithms). Generate a new $\stp'$ by (\ref{eg2: transition}) and let $j' = \stp'$
  \begin{equation}
    \label{algo: uncorrelated}
	(G_m)_i =  -(r_i + \g v_j-v_i), \qd (G_m)_{j'} =  \g(r_i + \g v_j-v_i), \qd (G_m)_l = 0, \qd \forall l \neq i, j.
  \end{equation}
  The uncorrelated sampling algorithm gives an unbiased estimation of the loss function (\ref{eg2:
    nb J 1}). However, it is impractical for model-free RL.
\item {\bf Sample-cloning:} 
  \begin{equation}
    \label{algo: correlated 2}
	(G_m)_i =  -(r_i + \g v_j-v_i), \qd (G_m)_j =  \g(r_i + \g v_j-v_i), \qd (G_m)_l = 0, \qd \forall l \neq i, j.
  \end{equation}
\item {\bf BFF-gradient:} Let $j' = \st + (\stpt - \stp)$. 
  \begin{equation}
    \label{algo: shift on gd}
	(G_m)_i =  -(r_i + \g v_j-v_i), \qd (G_m)_{j'} =  \g(r_i + \g v_j-v_i), \qd (G_m)_l = 0, \qd \forall l \neq i, j.
  \end{equation}
\item  {\bf BFF-loss:} Let $j' = \st + (\stpt - \stp)$. 
  \begin{equation}
    \label{algo: shift on loss}
    \begin{aligned}
	&(G_m)_i =  -\frac12(r_i + \g v_j-v_i) - \frac12(r_i + \g v_{j'}-v_i), \qd (G_m)_{j'} =  \frac\g2(r_i + \g v_j-v_i), \qd\\
	  &(G_m)_{j} =  \frac\g2(r_i + \g v_{j'}-v_i), \qd (G_m)_l = 0, \qd \forall l \neq i, j, j'.
    \end{aligned}
  \end{equation} 
\end{itemize}

%We would like to emphasize again the uncorrelated sampling algorithm is the only algorithm that its
%gradient is unbiased estimation of the loss function (\ref{eg2: nb J 1}). However, it is unrealistic
%for model-free RL.

%======================================================
\section{Error estimates}\label{sec: thm pf}

The aim of this section is to show that Algorithm \ref{algo: shiftonloss} (BFF-loss) \tcb{for the
  continuous state space RL with underlying dynamics \eqref{eq: discrete St 0}} is close to
Algorithm \ref{algo: uncor} (uncorrelated sampling) statistically.

\subsection{Difference between the objective functions}
Let us introduce
\begin{equation}
\label{eq: def of J}
\e(\th) :=\hJ(\th)-\J(\th).
\end{equation}
Notice that $\e(\th) = \E \up(\st;\th)$ with
\begin{equation}\label{eq: def of eps}
  \up(\st;\th) = \E\l[f(\st, \stp; \th)|\st\r]\E\l[f(\st,\st+\dstp; \th) - f(\st, \st+\dst; \th)|\st \r].
\end{equation}
The following lemma shows that if the values and derivatives of the drift, diffusion, and nonlinear
approximation are bounded, then the difference between the two objective functions $J(\th)$ and
$\hJ(\th)$ is smaller than $C\dt^2$, with the constant depending only on the size of $\a, \s, f$ and
their derivatives until the second order.

\begin{lemma}
\label{lemma: eps est}
For $\e, \up$ defined in (\ref{eq: def of J}), (\ref{eq: def of eps}), if $\ll \a^{(k)}(\cdot) \rl_{L^\infty}, 0\leq k
\leq 3, \ll\s^{(l)}(\cdot) \rl_{L^\infty}, 0\leq l\leq 4$ are bounded, and
$\ll\pt^i_{s_2}f(s_1,s_2;\th)\rl_{L^\infty_{s_1,s_2}}, 0\leq i \leq 5$ are also uniformly bounded
for any $\th$, then for all $\th$, one has
\begin{equation}  \label{eq: eps est}
  \begin{aligned}
    &\ll \up(s, \th)\rl_{L^\infty_s} \leq C\dt^2+o(\dt^2), \\
    &\e(\th) \leq C\dt^2+o(\dt^2),    
  \end{aligned}
\end{equation}
for some constant $C$ depending on $\ll \a^{(i)}(\cdot) \rl_{L^\infty}, \ll \s^{(i)}(\cdot)
\rl_{L^\infty}, \ll \pt^i_{s_2}f(s_1,s_2;\th)\rl_{L^\infty_{s_1,s_2,\th}}, 0\leq i\leq 2$.
\end{lemma}

The boundedness of the residual $f$ is followed by the boundedness of $R$ and $V$. Since we assume
that the state space $\S$ and the parameter space $\O$ are both compact, for parametric value
approximation, such as neural network, it is natural to assume $R,V$ are bounded.

\begin{proof}
See Appendix \ref{proof of eps est}.
\end{proof}

\subsection{Difference between the asymptotic distributions}
In this section and Section \ref{sec: finite diff}, we assume the optimization region of (\ref{eq:
  obj fcn}) is a bounded connected open subset $\O$ in $\R^d$. This assumption is to guarantee the
Poincare inequality. The updates of the parameter $\th$ of $J(\cdot)$ and $\hth$ of $\hJ(\cdot)$ by SGD
according to Algorithms \ref{algo: uncor} and \ref{algo: shiftonloss} can be approximated by
stochastic differential equations (SDEs) with $\eta = \frac{\tau}{M}$
\citep{li2017stochastic,hu2017diffusion}
\begin{equation*}
  \begin{aligned}
    &d\th_t = -\nb J(\th_t) dt + \sqrt{\eta}\Sig^{\frac12}(\th_t) dB_t;\\
    &d\hth_t = -\nb \l(J(\hth_t)+\e(\hth_t)\r) dt + \sqrt{\eta}\l(\Sig(\hth_t) + \Up(\hth_t)\r)^{\frac12} dB_t,
  \end{aligned}
\end{equation*}
where
\begin{equation*}
\begin{aligned}
  &\Sig(\th_t) = \Var\l[\frac12(\E\l[f(\st,\stp;\theta)|\st\r])^2\r];\\ &\Up(\hth_t) =
  \Var\l[\frac12(\E\l[f(\st,\stp;\theta)|\st\r])^2+\up\r] - \Sig(\hth).
\end{aligned}
\end{equation*}
Here $\Var$ represents the variance, and $\up$ is defined in (\ref{eq: def of eps}).

Therefore, the corresponding probability density functions $\p(t,\th)$, $\hp(t,\th)$ of $\th_t$, $\hth_t$ can
be described by \citep{pavliotis2014stochastic}
\begin{align}
    &\pt_tp(t,\th) = \nb\cdot\l[\l(\nb J\r)\p + \frac{\eta}{2}\nb\cdot\l(\Sig\p\r)\r]; \label{eq: eq for p}\\
    &\pt_t\hp(t,\th) = \nb\cdot\l[\l(\nb J + \nb \e\r)\hp + \frac{\eta}{2}\nb\cdot\l((\Sig+ \Up)\hp\r)\r], \label{eq: eq for hp}
\end{align}
with the same initial data $p(0,\th) = \hp(0,\th)$.
We use the reflecting boundary condition on $\pt \O$, 
\begin{equation}
\label{eq: ref bdy cond}
\begin{aligned}
    &\l.\l(\nb J p + \frac{\eta}{2}\nb\cdot(\Sig p) \r)\cdot \vec{n} \r\vert_{\pt\O}= 0,\\
    &\l.\l((\nb J+\nb \e) \hp + \frac{\eta}{2}\nb\cdot(\Sig + \Up \hp) \r)\cdot \vec{n} \r\vert_{\pt\O}= 0,\\
\end{aligned}
\end{equation}
which means that $\th$ will be reflected after hitting the boundary.

Besides, from the estimation we obtained in (\ref{eq: eps est}), we know that $\nb\e \leq O(\dt^2)$, and it is easy to see that $\Up \leq O(\dt^2)$ because
\begin{equation}
\label{eq: bd of Up}
\begin{aligned}
    \Up(\hth) =& \Sig(\hth) + \Var[\up] + \E\l[\up(\E\l[f(\st,\stp;\theta)|\st\r])^2\r] -\E\l[\up\r]\E\l[(\E\l[f(\st,\stp;\theta)|\st\r])^2\r] - \Sig(\hth) \\
    \leq&  O(\dt^4) + C(\hth)\dt^2  \leq C\dt^2 + o(\dt^2).
\end{aligned}
\end{equation}

\begin{assumption}
\label{ass: J hJ}
We assume both loss functions $J(\th)$ and $\hJ(\th)$ satisfy the following:
\begin{itemize}
    \item [-] $\ds \int e^{-J(\th)} d\th \leq \infty,$ and $\ds \int e^{-\hJ(\th)} d\th \leq \infty$.
    \item [-] The Frobenius norms of $G = \ll H(J) \rl_{L^\infty_\th}$ and $\h{G} =  \ll H(\hJ) \rl_{L^\infty_\th}$ are bounded by a constant $M$, where $H$ represents the Hessian and $L^\infty_\th$ is taken element-wisely to the matrix. 
\end{itemize}
\end{assumption}
The first assumption is to guarantee that the steady-state is well defined. The second assumption is
used to prove the boundedness of $\nb \hp$ .

\begin{theorem}
\label{thm: p_infty}
Assume $\Sig \sim O(1),\Up$ are both constants, then there exist steady-state distributions $\pinf,
\hpinf$ for (\ref{eq: eq for p}), (\ref{eq: eq for hp})
\begin{equation*}
    \pinf(\th) = \frac{1}{Z}e^{-\b J(\th)}, \qd \hpinf(\th) =
    \frac{1}{\hZ}e^{-\hb (J(\th)+\e(\th))}, \qd \b = \frac{2}{\eta\Sig}, \qd \hb = \frac{2}{\eta(\Sig + \Up)},
\end{equation*}
where $Z = \int e^{\frac{-2J}{\eta\Sig}} d\th$, $\hZ = \int e^{\frac{-2(J+\e)}{\eta(\Sig+\Up)}}
d\th$ are normalized constants. In addition,
\begin{equation*}
    \ll \frac{\hpinf}{\pinf} \rl_{L^\infty} \leq 1+O\l(\frac{\dt^2}{\eta}\r).
\end{equation*}
\end{theorem}

Theorem \ref{thm: p_infty} implies the following:
\begin{itemize}
\item [-] If the probability of the unbiased SGD (Algorithm \ref{algo: uncor}) converging to the
  optimal $\thst$ is $p$, then the probability of Algorithm \ref{algo: shiftonloss} is
  $\l(1+O(\frac{\dt^2}{\eta})\r)p$.
\item [-] In order to make Algorithm \ref{algo: shiftonloss} behaves similarly to the unbiased SGD,
  we have to let $\frac{\dt^2}{\eta}$ be small, which means we require $\dt$ to be small, but $\eta$
  to be larger than $\dt^2$. This makes sense because we are minimizing a biased objective function,
  so if we do the biased SGD too carefully, it will end up a worse minimizer of the true objective
  function.
\end{itemize}

\begin{proof}
See Appendix \ref{proof of p_infty}.
\end{proof}

\subsection{Difference between finite time distributions}\label{sec: finite diff}

Theorem \ref{thm: p_infty} is about the asymptotic behavior of the algorithm. Now we will study the
difference between the two algorithms at a finite time. The following Poincare inequality of the
probability measure $d\mu = \pinf d\th$ or $d\mu = \hpinf d\th$ in a bounded connected domain is valid
for any $\int h d\mu = 0$,
\begin{equation}
    \label{eq: poincare ineq}
    \int_\O \lv\nb h\rv^2d\mu \geq \lam\int_\O h^2 d\mu,
\end{equation}
with a constant $\lam$ depending on $d\mu$ and $\O$.  Based on the above Poincare inequality, we can
prove the difference between the two algorithms, as shown in the following theorem.  The difference
is measured in the following norm,
\begin{equation}
  \label{def of star norm}
  \ll h \rl^2_\stt = \int h^2 \frac{1}{\pinf} d\th,
\end{equation}
where $\pinf$ is defined in Theorem \ref{thm: p_infty}. 

\begin{theorem}
\label{thm: p_t}
Under Assumption \ref{ass: J hJ} and the assumptions in Theorem \ref{thm: p_infty}, one has,  
\begin{equation}
  \label{eq: p_t}
    \ll \p(t,\th) - \hp(t, \th)\rl_\stt^2 \leq \l(1+O\l(\frac{\dt^2}{\eta}\r)\r)O(\dt^2).
\end{equation} 
\end{theorem}

Theorem \ref{thm: p_t} tells us that the evolution of $\th$ in Algorithm \ref{algo: shiftonloss}
differs from the unbiased SGD within an error of $\l(1+O\l(\frac{\dt^2}{\eta}\r)\r)O(\dt^2)$ .

%Notice Theorem \ref{thm: p_t} are different because different norm is used, 

\begin{proof}
See Appendix \ref{proof of p_t}
\end{proof}

%\subsection{Main results and proof sketch}
%\label{sec: main results}

%======================================================
\section{Numerical Examples} \label{sec: numerics}

Several numerical examples are presented here to demonstrate the performance of the proposed
algorithms. Recall that the goal of the prediction problem (i.e., policy evaluation) is to
approximate $V(s)$ based on the trajectories.

%========
\subsection{Continuous state space} \label{sec: cont}
Consider a Markov decision process with a continuous state space $\S = \{s\in(0,2\pi]\}$. Suppose
  that the transition probability is prescribed implicitly via the following dynamics
\begin{equation}  \label{eg1: SDE}
  \begin{aligned}
    &\stp = \st+\a(\st)\dt + \s(\st)\sqrt{\dt}Z_\t,\\
    &\a(s) =  2\sin(s)\cos(s), \qd \s(s) = 1+\cos(s)^2,\qd \dt = 0.1.
  \end{aligned}
\end{equation}
The immediate reward function is $R(s) = (\cos(2s)+1)$, and the discount factor $\g$ is $0.9$.

\begin{figure}[htbp]
  \floatconts
      {fig: eg1}
      {\caption{Continuous state space:  approximation with a 3-layer neural network  with batch size 1000.}}
      {\includegraphics[width=0.9\textwidth]{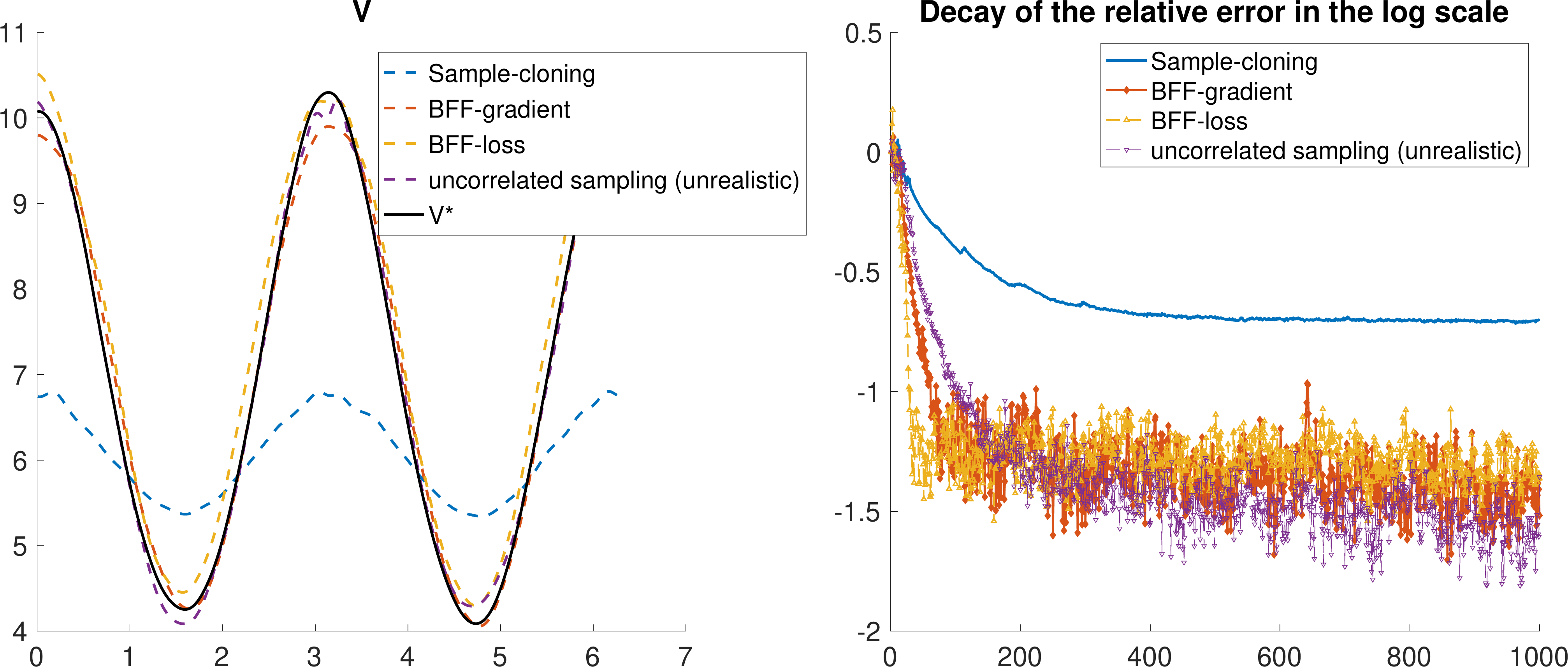}}
\end{figure}

A three-layer, fully connected neural network  is used to approximate the value
function $V(s;\th)$. The network has two hidden layers with $\cos$ as its activation function. Each hidden
layer contains $50$ neurons, i.e.,
\begin{equation}
  \label{eg1: NN}
  \begin{aligned}
    &V(s;\th) = V(x; \{w_i,b_i\}_{i=1}^3) = L_{w_3,b_3}\circ \cos\circ L_{w_2,b_2}\circ \cos\circ L_{w_1,b_1}((\cos s, \sin s)),\\
    &L_{w_i,b_i} (x) = w_ix+b_i, \qd w_i\in\R^{n_{i-1}\times n_i}, \qd b_i\in\R^{n_i},\qd n_0 = 2, n_1 = n_2 = 50, n_3 = 1.
  \end{aligned}
\end{equation}
The optimal $\thst$ is computed by Algorithms \ref{algo: uncor}-\ref{algo: shiftongd} based on a trajectory
$\{\st\}_{m=1}^{10^6}$ generated from \eqref{eg1: SDE}. The function $f$ in Algorithms \ref{algo: uncor}-\ref{algo: shiftongd} refers to
\begin{equation}
\label{f}
f(\st,\stp, \th) = R(\st) +\g V(\stp; \th) - V(\st; \th)
\end{equation}
in the value evaluation. The learning rate $\tau$ and the batch size $M$ are set to be
\begin{equation}
  \label{para}
  \tau = 0.1, \qd M = 1000.
\end{equation}
\tcb{Once the whole trajectory is recorded, we perform a random permutation and use batches of size
  $M=1000$ for training.} In each experiment, the SGD algorithm runs for a single epoch with the
same initialization $\th_0$ and randomly-permuted trajectory. The error $e_k$ at each step $k$ is
defined as the squared $L_2$ norm $\| V(\cdot,\th_k)-\vst(\cdot)\|_2$, where the reference solution
$\vst(s)$ is computed by running Algorithm \ref{algo: uncor} for $10$ epochs based on a longer
trajectory $\{\st\}_{m=1}^{10^7}$, with hyper-parameters $\tau = 0.01$ and $M = 1000$. The left plot
of Figure \ref{fig: eg1} shows the final $V(s, \th)$ obtained by four different methods, while the
relative errors $\log_{10}({e_k}/{e_0})$ in the log scale are summarized in the right plot.  The
plots show that the Sample-cloning algorithm introduces a rather large error while the BFF
algorithms are much closer to the (impractical) uncorrelated sampling algorithm.

\tcb{ Next, We compare the BFF algorithm with the primal-dual algorithm, i.e., the mini-batch SGD
  method for the objective function \eqref{obj nn pd}. The algorithm reads,
\begin{equation}
\label{algo: PD nn}
\begin{aligned}
  &\o_{k+1} =  \o_k + \frac{\beta}M \sum_{\st\in B_k} \l(f(\st,\stp,\th_k)\nb_\o y(\st;\o_k) - y(\st;\o_k)\nb_\o y(\st;\o_k)\r),\\
  &\th_{k+1} = \th_k - \frac{\tau}M \l(\sum_{\st\in B_k}\nb_\th f(\st,\stp,\th_k) y(\st;\o_{k+1}) \r),
\end{aligned}
\end{equation}
with $f,\tau, M$ as in \eqref{f}, \eqref{para} and $\beta = 0.5$. Each simulation uses the same
initialization for $\th_0$ and the same random permutation to the trajectory. The results are
summarized in Figure \ref{fig: eg1_compare}. The five dash curves in Figure \ref{fig: eg1_compare}
are the relative errors of the primal-dual algorithm with five different random initializations of
$\o_0$, and among them, only one simulation converges within $1000$ steps. This is because the
algorithm might be trapped at a suboptimal stationary point in the nonlinear approximation setting.
}

\begin{figure}[htbp]
  \floatconts
      {fig: eg1_compare}
      {\caption{Continuous state space:  approximation with a 3-layer neural network  with batch size 1000.}}
      {\includegraphics[width=0.5\textwidth]{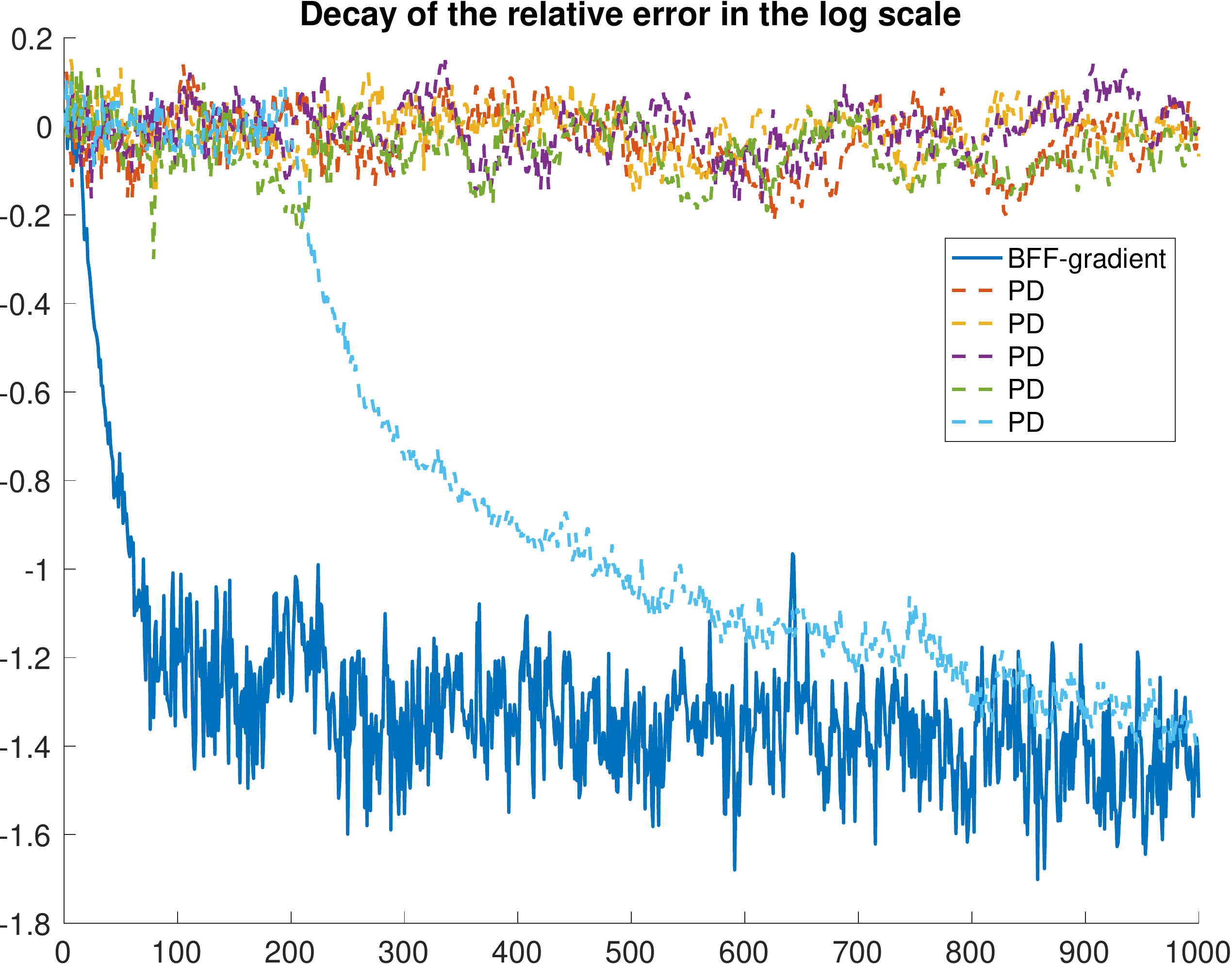}}
\end{figure}

%========
\subsection{Discrete state space}

Consider a Markov decision process with a discrete state space $\S = \{i = 0, \cdots, n-1\}$ for $n = 32$. The
transition matrix is the following,
\begin{equation}  \label{eg2: transition}
  \begin{aligned}
    \P_{i,i+1} = \frac{1}{2} - \frac{1}{5} \sin\frac{2\pi i}{n}, \qd \text{when }i = n-1, i+1 = 0;\\
    \P_{i,i-1} = \frac{1}{2} + \frac{1}{5} \sin\frac{2\pi i}{n}, \qd \text{when }i = 0, i-1 = n-1.\\
  \end{aligned}
\end{equation}
The immediate reward function is $r\in \R^{n}$ with $r_i = 1+ \cos\frac{2\pi i}{n}$, and the discount
rate is $\g = 0.9$. The value function $\vst\in\R^{n}$ satisfies the following Bellman equation,
\begin{equation}  \label{eg2: reference solu}
  \begin{aligned}
    \vst = \T(\vst) = r +\g P \vst,
    %, \qd \text{where}\qd P_{i,i'} = \P(\st = s_i, \stp = s_{i'}).
  \end{aligned}
\end{equation}
and can be solved from \eqref{eg2: reference solu}. The numerical results are carried out
both with the neural network approximation and in the tabular setting.

%In Section \ref{sec: discrete nonlinear}, we use neural network to approximate the value function,
%while in Section \ref{sec: discrete tabular}, we use tabular form to represent the value function.

%\subsubsection{Neuron Network Approximation}
%\label{sec: discrete nonlinear}

\paragraph{Neural network approximation.}
The same neural network architecture is used as in (\ref{eg1: NN}) with input $s = \frac{2\pi i}{n}$
for approximating the value function. We run Algorithms \ref{algo: uncor}-\ref{algo: shiftongd} and
 the primal-dual algorithm (\ref{algo: GTD}) to approximate $\thst$ based on a
trajectory $\{\st\}_{m=1}^{T}, T = 4\times10^6$ simulated from (\ref{eg2: transition}).

%The primal-dual algorithm is the following,\begin{equation}\label{algo: PD
%vector}\begin{aligned}&y_{k+1} = y_k + \beta M^{-1} \sum_{\st\in B_k} \l(f(\st,\stp,\th_k) -
%\Phi(\st)^\top y_k\r)\Phi(\st);\\&\th_{k+1} = \th_k - \frac{\tau}M \l(\sum_{\st\in B_k}\nb_\th
%f(\st,\stp,\th_k) \Phi(\st)^\top y_{k+1}\r),\end{aligned}\end{equation}where $\Phi(s_i) = e_i \in
%\R^N$ is a vector where all coordinates are zero except the $i$-th coordinate is
%$1$. $M\in\R^{N\times N} = \text{diag}(M_1,\cdots, M_N)$ with $M_i = \#\{\st \in B_k: \st = i\}$.

\begin{figure}[htbp]
  \floatconts {fig: eg2 batch1} {\caption{Discrete state space: approximation with a 3-layer neural
      network with batch size 1.}}  {\includegraphics[width=0.9\textwidth]{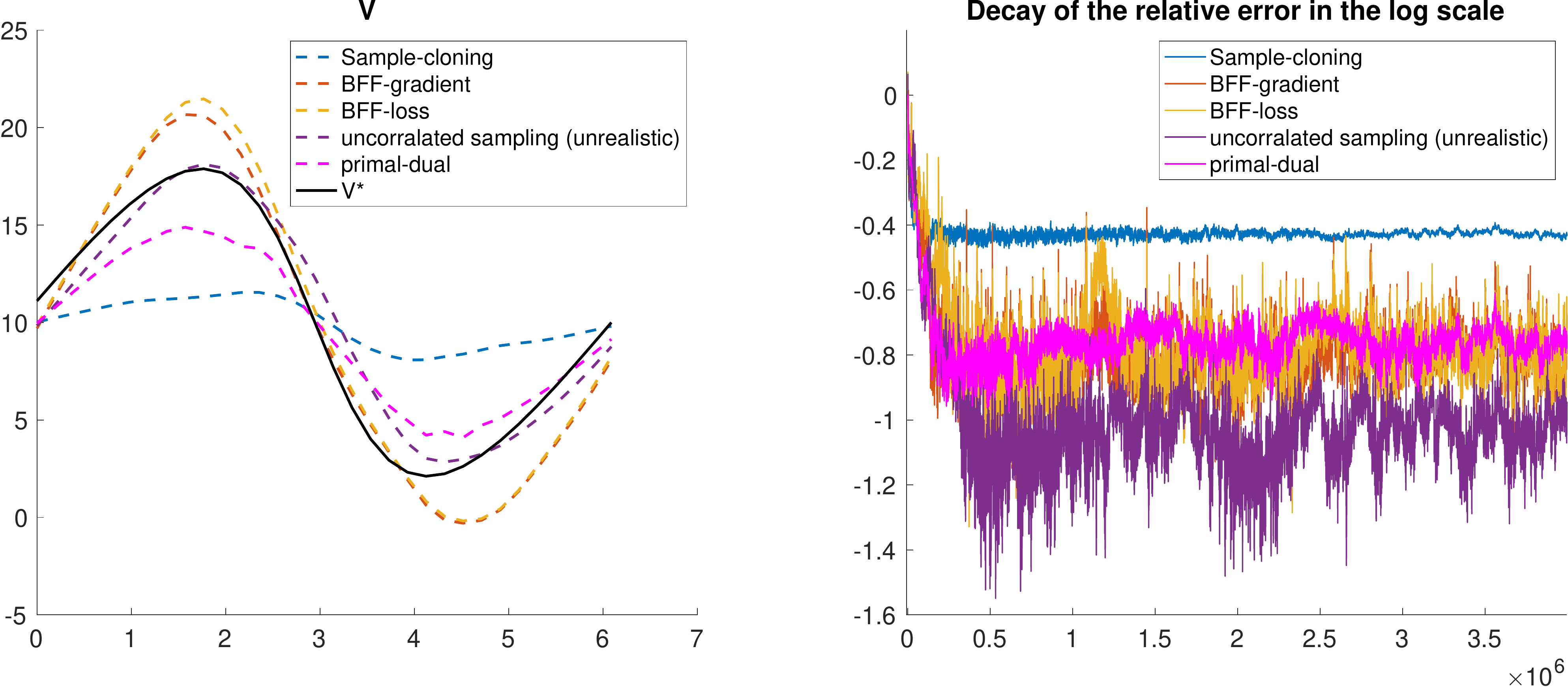}}
\end{figure}

Figure \ref{fig: eg2 batch1} summarizes the results obtained from SGD with a single batch at
the learning rate $\tau = 0.001$ and $\beta = 0.1$ for the primal-dual algorithm \eqref{algo: GTD}. The
same initialization for $\th_0$ and random permutation are used for all simulations. The
initialization for $y_0$ in the primal-dual algorithm is the zero vector.  \color{black} The
relative errors $e_k/e_0$, with $e_k$ defined as
\[
e_k=\sqrt{\sum_{i=0}^{n-1}\l(V\l(\frac{2\pi i}{n},\th_k\r) - \vst_i\r)^2},
\]
are shown in the log scale on the right. The plots of Figure \ref{fig: eg2 batch1} demonstrate that
the performance of the two new algorithms (BFF-loss and BFF-gradient) are very similar. Although
they are less accurate than the uncorrelated case, \tcb{the performance is much better than the
  Sample-cloning algorithm and comparable to the primal-dual algorithm.}

%\begin{figure}[htbp]
%  \floatconts
%      {fig: eg2 batch1000}
%      {\caption{Discrete state space: approximation with a 3-layer neural network with batch size 1000.}}
%      {\includegraphics[width=1\textwidth]{eg2_batch1000.pdf}}
%\end{figure}

%In Figure \ref{fig: eg2 batch1}, the BFF algorithm exhibit a larger oscillation since the small batch size ($M=1$) results in more stochasticity in the training dynamics compared to the large match size ($M=1000$).

\paragraph{Tabular setting.}
The BFF algorithms proposed in Section \ref{sec: tabular} are used to approximate $\vst$ in the
tabular form. In the experiments, $\tau = 0.1$ and the SGD is run for $5$ epochs. \tcb{Again, we
  compare the BFF with the primal-dual algorithm \eqref{algo: GTD} (equivalent to the SCGD algorithm
  in this case) with $\b = 0.5$.  In the tabular setting, $V(s;\th) = \Phi(s)^\top \th$ with
  $\th\in\R^N$ and $\Phi(s_i)=e_i\in\R^N$, where $\{e_i\}$ are the standard basis vectors of
  $\R^N$.}  The results summarized in Figure \ref{fig: eg2 tabular} shows that the BFF-loss
algorithm converges slightly faster than the primal-dual/SCGD algorithm, and both significantly
better than the Sample-cloning algorithm. Comparing with Figure \ref{fig: eg2 batch1}, it seems
that the neural network approximation results in significantly faster error decay at the initial
stage of the training.

\begin{figure}[htbp]
  \floatconts
      {fig: eg2 tabular}
      {\caption{Discrete state space: tabular approximation.}}
      {\includegraphics[width=0.9\textwidth]{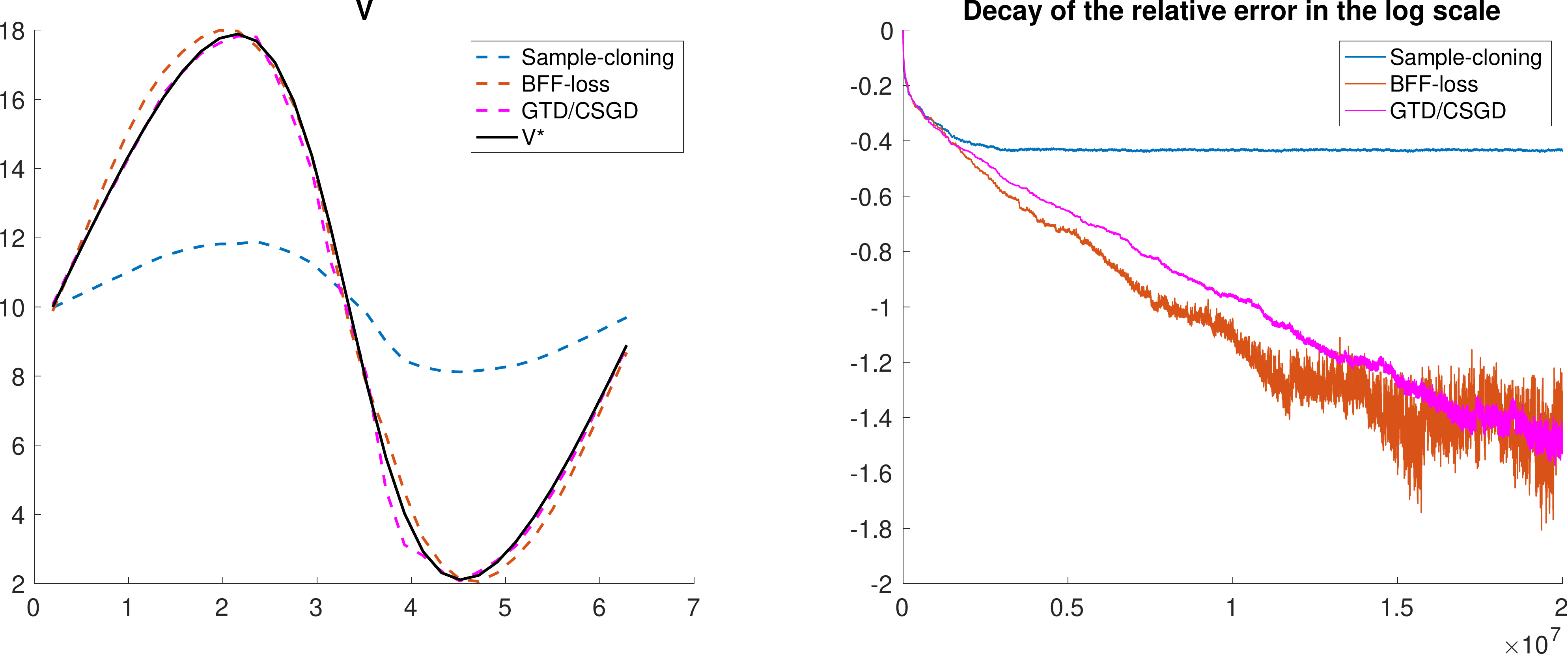}}
\end{figure}

\paragraph{Summary.}
The numerical experiments suggest that BFF algorithms significantly outperform the Sample-cloning
algorithm, as the theory predicted. For the continuous state space setting, the BFF algorithm
performs better than primal-dual algorithms. For the discrete state space setting, BFF is comparable
to the primal-dual algorithm and SCGD.  \color{black}
%======================================================

% Acknowledgments---Will not appear in anonymized version
\acks{ The work of L.Y. and Y.Z. is partially supported by the U.S. Department of Energy, Office of
  Science, Office of Advanced Scientific Computing Research, Scientific Discovery through Advanced
  Computing (SciDAC) program. The work of L.Y. is also partially supported by the National Science
  Foundation under award DMS-1818449. L.Y. thanks Mohammad Ghavamzadeh, Yuandong Tian, and Amy Zhang
  for helpful discussions.  }

%======================================================
\clearpage
\def\cprime{$'$}

%======================================================
\appendix 
\section{Proof of Lemma \ref{lemma: eps est} }
\label{proof of eps est}
\begin{proof}
Let 
\begin{equation*}
    \d(\st,\th) = \E\l[f(\st,\st+\dstp; \th) - f(\st, \st+\dst; \th)|\st \r],
\end{equation*}
then 
\begin{equation}
\label{eq: eps 1}
    \up(\st,\th) = \E\l[f(\l.\st,\stp;\th)\r\vert \st\r] \d(\st,\th);\qd \e(\th) = \frac{1}{2}\E\, \up(\st,\th).
\end{equation}
We first estimate the term $\d(\st,\th)$. By Taylor expansion,
\begin{equation*}
\begin{aligned}
    &f(\st, \st+\dstp; \th) - f(\st, \st+\dst; \th) \\
    =& \l[f(\st, \st+\dstp; \th) - f(\st, \st; \th)\r] - \l[f(\st, \st+\dst - f(\st, \st; \th)\r] \\ 
    =& \pt_{s_2}f(\st,\st;\th)\l(\dstp  - \dst\r)+  \frac{1}{2}\pt^2_{s_2}f(\st,\st;\th)\l(\dstp^2 - \dst^2\r) \\
    &+ \frac{1}{6}\pt^3_{s_2}f(\st,\st;\th)\l(\dstp^3 - \dst^3\r)+ \frac{1}{24}\pt^4_{s_2}f(\st,\st;\th)\l(\dstp^4 - \dst^4\r) \\
    &+  \frac{1}{120}\l(\pt^5_{s_2}f(\st,\st+s';\th)\dstp^5 - \pt^5_{s_2}f(\st,\st+s'';\th)\dst^5\r),
\end{aligned}
\end{equation*}
for some $s'\in(0, \dstp), s''\in(0,\dst)$. By the definition of $\dst$ in (\ref{eq: def of ds}), we
have $\dst^5, \dstp^5\leq o(\dt^2)$, so the last term of the above equation is of order
$o(\dt^2)$. Therefore, one has,
\begin{equation}
\label{eq: d_th 1}
\begin{aligned}
    &\d(\st,\th) = \sum_{i=1}^4 \frac{1}{i!}\pt^i_{s_2}f(\st,\st;\th) \E\l[\dst^i-\dstp^i|\st\r] + o(\dt^2).
\end{aligned}
\end{equation}
The Taylor expansion of $\a(\stp), \s(\stp)$ can be represented by
\begin{equation}
\label{eq: dstp at st}
\begin{aligned}
    &\a(\stp) = \a(\st) + \a' (\st)\dst + \frac{1}{2}\a''(\st)\dst^2 + o(\dt),\\
    &\s(\stp) = \s(\st) + \s'(\st) \dst+ \frac{1}{2}\s''(\st)\dst^2+\frac{1}{6}\s'''(\st)\dst^3 + o(\dt^{3/2}) ,
\end{aligned}
\end{equation}
which gives, 
\begin{equation*}
\begin{aligned}
    &\dstp - \dst = \l(\a(\stp) - \a(\st)\r)\dt + \s(\st)Z_{m+1}\sdt - \s(\st)Z_m\sdt\\
    =&\l(\a' (\st)\dst + \frac{1}{2}\a''(\st)\dst^2\r)\dt+o(\dt^2)  \\
    &+ \l(\s(\st) + \s'(\st) \dst+ \frac{1}{2}\s''(\st)\dst^2+\frac{1}{6}\s'''(\st)\dst^3\r)Z_{m+1}\sdt - \s(\st)Z_m\sdt+o(\dt^2).
\end{aligned}
\end{equation*}
Since $\E[h(\st)Z_m|\st] = \E [g(\st)Z_{m+1}|\st] = 0$ for any functions $h, g$, the last line of the
above equation vanishes after taking conditional expectation on $\st$. This implies,
\begin{equation}
\label{eq: d_th 2}
\begin{aligned}
  &\E\l[\dstp - \dst|\st\r]
  = \E\l[\l(\a' (\st)\dst + \frac{1}{2}\a''(\st)\dst^2\r)\dt|\st\r]+o(\dt^2) \\
  =&\E\l[\l(\a'(\a\dt+\s\sdt Z_m) + \frac{1}{2}\a''(\a^2\dt^2+2\a\s \dt^{3/2}Z_m + \s^2\dt Z_m^2 )\r)\dt|\st\r]+o(\dt^2)\\
  =& \a'\a\dt^2 + \frac{1}{2}\a'' \s^2\dt^2  +o(\dt^2).
\end{aligned}
\end{equation}
Here $\a,\a',\s$ refers to the function's value at $\st$, similar for $\s',\s'',
\pt_{s_2}^if$. We omit $(\st)$ when the functions has its value at $\st$.

Using (\ref{eq: dstp at st}), one can estimate $\dstp^i$ for $i=2,3,4$ as follows, 
\begin{equation*}
\begin{aligned}
    \dstp^2 = &\a(\stp)^2\dt^2+\s(\stp)^2Z_{m+1}^2\dt + 2\a(\stp)\s(\stp)Z_{m+1}\dt^{3/2}\\
    =& \a^2\dt^2 + \l(\s^2 + (\s')^2 \dst^2 +2\s\s'\dst+ \s\s''\dst^2\r)Z_{m+1}^2\dt \\
    &+2\l(\a\s + \a\s'\dst+\a'\s\dst\r)Z_{m+1}\dt^{3/2}+ o(\dt^2);\\
    \dstp^3 = &3\a(\stp)\s^2(\stp)Z_{m+1}^2\dt^2 + \s^3(\stp) Z_{m+1}^3\dt^{3/2} + o(\dt^2)\\
    =&3\a\s^2Z^2_\tp\dt^2+ \s^3(\stp) Z_{m+1}^3\dt^{3/2} + o(\dt^2);\\
    \dstp^4 =&  \s^4(\stp) Z_{m+1}^4\dt^2+ o(\dt^2) =\s^4 Z_{m+1}^4\dt^2 + o(\dt^2).
\end{aligned}
\end{equation*}
Therefore, 
\begin{equation}
\label{eq: d_th 3}
\begin{aligned}
    \E\l[\dstp^2 - \dst^2|\st\r] 
    = &\E\l[ \l(\s'^2 \dst^2 +2\s\s'\dst+ \s\s''\dst^2\r)\dt|\st\r] + o(\dt^2)\\
    =&\E\l[\l( \l(\s'^2 + \s\s''\r) \l(\s^2Z_m^2\dt+o(\dt)\r) + 2\s\s'(\a\dt+\s Z_m\sdt) \r)\dt|\st\r] + o(\dt^2)\\
    =&\l(\s'^2 + \s\s''\r)\s^2 \dt^2 + 2\s\s'\a\dt^2+ o(\dt^2);\\
    \E\l[\dstp^3 - \dst^3|\st\r] = &o(\dt^2);\\
    \E\l[\dstp^4 - \dst^4|\st\r]  = &  o(\dt^2).
\end{aligned}
\end{equation}
Hence, by inserting (\ref{eq: d_th 2}) and (\ref{eq: d_th 3}) into (\ref{eq: d_th 1}) gives,
\begin{align}
    \d(\th) = \l[\pt_{s_2}f\l(\a'\a + \frac{1}{2}\a'' \s^2\r) +\pt^2_{s_2}f \l(\s'^2\s^2 + \s\s''\s^2 + 2\s\s'\a\r)\r]\dt^2 + o(\dt^2).
\end{align}
As defined in (\ref{eq: def of eps}) and (\ref{eq: eps 1}), the completion of the proof is followed by,
\begin{equation*}
\begin{aligned}
    &\up =\frac12 \E [f(\st,\stp;\th)|\st] \d \leq C\dt^2 + o(\dt^2);\\
    &\e = \E\up \leq C\dt^2 + o(\dt^2).
\end{aligned}
\end{equation*}
\end{proof}

\section{Proof of Theorem \ref{thm: p_infty}}
\label{proof of p_infty}
\begin{proof}
We first observe
\begin{equation*}
\begin{aligned}
  \frac{\hpinf}{\pinf} = \frac{Z}{\hZ} \frac{e^{-\hb(J(\th)+\e(\th))}}{e^{-\b J(\th)}} =
  \frac{Z}{\hZ}e^{-(\hb - \b)J(\th)}e^{-\hb\e(\th)} =
  \frac{Z}{\hZ}e^{\frac{\Up}{\eta}\frac{2J(\th)}{\Sig(\Sig+\Up)}}e^{-\frac{\e(\th)}{\eta}\frac{2}{(\Sig+\Up)}}.
\end{aligned}
\end{equation*}
By the fact that $\Up \leq O(\dt^2), \e \leq O(\dt^2)$, we have
\begin{equation*}
  \ll \frac{\hpinf}{\pinf} \rl_{L^\infty} \leq \frac{Z}{\hZ}e^{O\l(\frac{\dt^2}{\eta}\r)}\leq
  \frac{Z}{\hZ}\l(1+O\l(\frac{\dt^2}{\eta}\r)\r).
\end{equation*}
Similarly, it is easy to see that $\frac{Z}{\hZ} \sim (1+O(\dt^2/\eta))$ because, 
\begin{equation*}
\begin{aligned}
  \frac{Z}{\hZ} = \frac{\int_{\O} e^{-\b J}d\th}{\int_{\O}e^{-\b J}e^{-(\hb-\b)J}e^{-\hb\e} d\th} =
  \frac{\int_{\O} e^{-\b J}d\th}{\int_{\O}e^{-\b J} \l(1+O\l(\frac{\dt^2}{\eta}\r)\r)^2 d\th} =
  1+O\l(\frac{\dt^2}{\eta}\r).
\end{aligned}
\end{equation*}
Therefore, 
\begin{equation*}
  \ll \frac{\hpinf}{\pinf} \rl_{L^\infty} \leq  \l(1+O\l(\frac{\dt^2}{\eta}\r)\r)^2 \leq 1+O\l(\frac{\dt^2}{\eta}\r).
\end{equation*}
\end{proof}

\section{Proof of Theorem \ref{thm: p_t}}
\label{proof of p_t}
\begin{proof}
Letting $h(t,\th) = \hp(t, \th) - \p(t,\th)$ and subtracting (\ref{eq: eq for p}) from (\ref{eq: eq
  for hp}) leads to
\begin{equation*}
  \begin{aligned}
    \pt_t h =&\nb\cdot\l[(\nb J+\nb \e) \hp + \frac{1}{\hb} \nb\hp\r]-\nb\cdot\l[\nb J \p+ \frac{1}{\b}\nb \p\r] \\
    =& \nb\cdot\l[\nb J h+ \frac{1}{\b}\nb h\r] + \nb\cdot\l[\nb \e \hp + \l(\frac{1}{\hb} - \frac{1}{\b}\r)\nb\hp\r]\\
    =& \nb\cdot\l[\pinf \nb\l(\frac{h}{\pinf}\r)\r] + \nb\cdot\l[\nb \e \hp + \l(\frac{1}{\hb} - \frac{1}{\b}\r)           \nb\hp\r].
  \end{aligned}
\end{equation*}
Multiply $\frac{h}{\pinf}$ to the above equation, then integrate it over $\th$, one has, 
\begin{equation*}
\begin{aligned}
    \frac12\pt_t \ll h\rl^2_\stt 
     =& \l.\frac{h}{\pinf}\l((\nb J+\nb \e) \hp + \frac{1}{\hb} \nb\hp\r)\cdot \vec{n} \r\vert_{\pt \O} - \l.\frac{h}{\pinf}\l(\nb J \p+ \frac{1}{\b}\nb \p\r)\cdot \vec{n} \r\vert_{\pt \O}\\
     &-\int\l[ \nb\l(\frac{h}{\pinf}\r)\r]^2 \pinf d\th - \int  \l[\nb \e \hp + \l(\frac{1}{\hb} - \frac{1}{\b}\r)  \nb\hp\r]\cdot \nb\l(\frac{h}{\pinf}\r)d\th.
\end{aligned}
\end{equation*}
The first two terms on the RHS vanish because of the reflecting boundary condition (\ref{eq: ref
  bdy cond}). Since $\nb\e,\Up \leq O(\dt^2)$ have been shown in Lemma \ref{lemma: eps est} and
(\ref{eq: bd of Up}), this leads to the two coefficients of the last term can be bounded by
$\ll\nb\e\rl_{L^\infty} \leq C_1\dt^2$, $\frac{1}{\hb} - \frac{1}{\b} = \eta\Up/2\leq
C_2\eta\dt^2$. Therefore, applying Young's inequality to the last term gives,
\begin{equation*}
\begin{aligned}
    &\int  \l[\nb \e \hp + \l(\frac{1}{\hb} - \frac{1}{\b}\r)  \nb\hp\r]\cdot \nb\l(\frac{h}{\pinf}\r)d\th\\
     \leq& \ll \nb \e\rl_{L^\infty}\int  \lv   \hp \cdot  \nb\l(\frac{h}{\pinf}\r)\rv d\th +\l(\frac{1}{\hb} - \frac{1}{\b}\r)  \int  \lv \nb\hp\cdot \nb\l(\frac{h}{\pinf}\r)\rv d\th\\
     \leq& \frac12 C_1\dt^2\l(\ll \nb  \hp  \rl_\stt^2 + \int \l[\nb\l(\frac{h}{\pinf}\r)\r]^2 \pinf  d\th\r) + \frac12 C_2\eta\dt^2\l(  \ll \nb\hp \rl^2_\stt + \int \l[\nb\l(\frac{h}{\pinf}\r)\r]^2\pinf d\th\r).
\end{aligned}
\end{equation*}
The third term can be bounded according to the Poincare Inequality (\ref{eq: poincare ineq}), thus
one has
\begin{equation}
\label{eq: p_t 1}
\begin{aligned}
    &\frac12\pt_t \ll h\rl^2_\stt \\
     \leq & -\frac{\lam}{2} \ll h \rl^2_\stt -\frac{1}{2}\int\l[ \nb\l(\frac{h}{\pinf}\r)\r]^2 \pinf d\th+ \frac{\dt^2}{2}\l(C_1\ll \hp \rl^2_\stt + C_2\eta\ll \nb\hp \rl^2_\stt + (C_1 + C_2\eta)\int\l[ \nb\l(\frac{h}{\pinf}\r)\r]^2 \pinf d\th\r)\\
     \leq & -\frac{\lam}{2} \ll h \rl^2_\stt + \frac{\dt^2}{2}\l(C_1\ll \hp \rl^2_\stt + \eta C_2\ll \nb\hp \rl^2_\stt \r).
\end{aligned}
\end{equation}
Since we only consider the case when $\dt<\!\!<1$, so the coefficient of $\int\l[
  \nb\l(\frac{h}{\pinf}\r)\r]^2 \pinf d\th$, $-(1-\dt^2(C_1+C_2\eta))/2$, is always negative, which gives the
last inequality of the above estimates.

Therefore, as long as $\ll \hp \rl^2_\stt, \ll \nb\hp \rl^2_\stt $ are bounded, we can bound
$\ll h \rl^2_\stt$ from (\ref{eq: p_t 1}). We prove the boundedness of $\ll \hp \rl^2_\stt, \ll
\nb\hp \rl^2_\stt$ in Lemma \ref{lemma: bd of hp} and Lemma \ref{lemma: bd of nb_hp} in Appendix
\ref{apdx: boundedness} and \ref{apdx: boundedness 2}.  Then, we can bound the last term of (\ref{eq: p_t
  1}) by $$\ds \frac{\dt^2}{2}\l(C_1\ll \hp \rl^2_\stt + \eta C_2\ll \nb\hp
\rl^2_\stt \r) \leq \frac{C}{2}\ll\frac{\hpinf}{\pinf} \rl_{L^\infty}\dt^2,$$ for some constant
$C$. Hence from (\ref{eq: p_t 1}), we have,
\begin{equation*}
\begin{aligned}
    &\pt_t\l( e^{\lam t}\ll h\rl^2_\stt\r) 
     \leq  e^{\lam t} \l(C\ll\frac{\hpinf}{\pinf} \rl_{L^\infty} \dt^2\r), \\
    &e^{\lam t}\ll h\rl^2_\stt - \ll h(0)\rl^2_\stt \leq \frac{1}{\lam}(e^{\lam t} - 1)C\ll\frac{\hpinf}{\pinf} \rl_{L^\infty}\dt^2.
\end{aligned}
\end{equation*}
Since $\p(0,\th) = \hp(0,\th)$, $h(0,\th) = 0$. Therefore,
\begin{equation*}
\begin{aligned}
    &\ll h\rl^2_\stt  \leq \frac{1}{\lam}(1- e^{-\lam t})C\ll\frac{\hpinf}{\pinf} \rl_{L^\infty}\dt^2 \leq \l(1+O\l(\frac{\dt^2}{\eta}\r)\r)O(\dt^2) .
\end{aligned}
\end{equation*}

\end{proof}

\section{}
\label{apdx: boundedness}
\begin{lemma}
\label{lemma: bd of hp}
The solution to (\ref{eq: eq for hp}) is bounded $$\ds\ll \hp \rl^2_\stt \leq
C\ll\frac{\hpinf}{\pinf} \rl_{L^\infty},$$with some constant $C$ related to the initial data.
\end{lemma}
\begin{proof}
We first prove that the difference of $\hp$ and $\hpinf$ is exponentially decay. Define norm $\ll g
\rl^2_\hst$ as follows,
\begin{equation*}
    \ll g \rl^2_\hst = \int g^2\frac{1}{\hpinf} d\th,
\end{equation*}
where $\hpinf$ is defined in Theorem \ref{thm: p_infty}. 
Let $g = \hp - \hpinf$, then $g$ satisfies, 
\begin{equation}
\label{eq: g}
    \pt_tg = \nb\cdot\l[\hpinf\nb\l(\frac{g}{\hpinf}\r)\r].
\end{equation}
Multiplying $\frac{g}{\hpinf}$, and integrating it over $\th$, after integration by parts, one has
\begin{equation*}
    \frac12\pt_t\ll g \rl^2_\hst = -\int \hpinf\l[\nb\l(\frac{g}{\hpinf}\r)\r]^2d\th \leq - \lam \ll g \rl_\hst^2,
\end{equation*}
where the last inequality follows from the Poincare inequality (\ref{eq: poincare ineq}) and the
fact that $\int g d\th = 1-1 = 0$. Solve the above ODE, one has,
\begin{equation}
\label{eq: bd of hp 0}
\begin{aligned}
    \ll g(t) \rl^2_\hst \leq e^{-2\lam t} \ll g(0) \rl^2_\hst
\end{aligned}
\end{equation}
Therefore, one can bound $\hp$ by
\begin{equation}
\label{eq: bd of hp 1}
\begin{aligned}
    \ll \hp \rl^2_\stt =& \ll \hp \sqrt{\frac{\hpinf}{\pinf}} \rl^2_\hst \leq \ll\frac{\hpinf}{\pinf} \rl_{L^\infty} \ll \hp \rl^2_\hst\leq \ll\frac{\hpinf}{\pinf} \rl_{L^\infty}\l(\ll \hp - \hpinf \rl^2_\hst + \ll \hpinf \rl^2_\hst\r) \\
    \leq&\ll\frac{\hpinf}{\pinf} \rl_{L^\infty}\l(e^{-2\lam t}\ll g(0) \rl^2_\hst + 1\r) .   
\end{aligned}
\end{equation}

\end{proof}

\section{}
\label{apdx: boundedness 2}
\begin{lemma}
\label{lemma: bd of nb_hp}
The gradient of the solution to (\ref{eq: eq for hp}) is bounded $$\ds\ll \nb\hp \rl^2_\stt \leq
C\ll\frac{\hpinf}{\pinf} \rl_{L^\infty},$$ with some constant $C$ related to the initial data.
\end{lemma}
\begin{proof}
Similar to (\ref{eq: bd of hp 1}) in the proof of Lemma \ref{lemma: bd of hp}, it is sufficient to
prove this Lemma if we get the estimation for $\ll \nb g(t) \rl^2_\hst$ with $g = \hp - \hpinf$,
because
\begin{align}
\label{eq: bd of nbhp 1}
    & \ll \nb\hp \rl^2_\stt \leq \ll\frac{\hpinf}{\pinf} \rl_{L^\infty}\l(\ll \nb g \rl^2_\hst + 1\r)
\end{align}

First notice that reflecting boundary condition also holds for $\pt_{\th_i}\hp$, that is,
\begin{align}
    &\l.\pt_{\th_i}\l((\nb J+\nb \e) \hp + \frac{\eta}{2}\nb\cdot(\Sig + \Up \hp) \r)\cdot \vec{n} \r\vert_{\pt\O}= 0,\nonumber
\end{align}
Then take $\pt_{\th_i}$ to (\ref{eq: g}), one has, 
\begin{equation*}
    \pt_t\pt_{\th_i} g = \nb\cdot\l[\hpinf\nb\l(\frac{\pt_{\th_i} g}{\hpinf}\r)\r] + \nb\cdot\l[\nb\l(\pt_{\th_i} \hJ\r) g\r].
\end{equation*}
Multiplying $\frac{\pt_{\th_i}g}{\hpinf}$, and summing it over $i$, integrating it over $\th$ gives
\begin{equation*}
\begin{aligned}
    \frac12\pt_t\ll \nb g \rl^2_\hst =& \l.\pt_{\th_i}\l((\nb J+\nb \e) g + \frac{\eta}{2}\nb\cdot(\Sig + \Up g) \r)\cdot \vec{n} \r\vert_{\pt\O}\\
    &-\sum_i \int \hpinf\l[\nb\l(\frac{\pt_{\th_i}g}{\hpinf}\r)\r]^2d\th - \sum_i\int \l[\nb\l(\pt_{\th_i} \hJ\r) g\r]\cdot \nb\l(\frac{\pt_{\th_i}g}{\hpinf}\r) d\th\\
    \leq& -\frac{\lam}{2}\sum_i\ll\pt_{\th_i}g\rl^2_\hst -\frac{1}{2}\sum_i \int \hpinf\l[\nb\l(\frac{\pt_{\th_i}g}{\hpinf}\r)\r]^2d\th \\
    &+\frac{1}{2}\sum_i\int \l[\nb\l(\pt_{\th_i} \hJ\r) g\r]^2\frac{1}{\hpinf} d\th + \frac{1}{2}\sum_i \int \hpinf\l[\nb\l(\frac{\pt_{\th_i}g}{\hpinf}\r)\r]^2d\th\\
    \leq &-\frac{\lam}{2}\ll\nb g\rl^2_\hst  + \frac{1}{2}\sum_i\int \l[\nb\l(\pt_{\th_i} \hJ\r) g\r]^2\frac{1}{\hpinf} d\th \leq -\frac{\lam}{2}\ll\nb g\rl^2_\hst  + \frac{M}{2} \ll g\rl_\hst^2,    
\end{aligned}
\end{equation*}
where the second assumption in Assumption \ref{ass: J hJ} is applied to obtain the last
inequality. Use the estimation of $\ll g \rl^2_\hst$ in (\ref{eq: bd of hp 0}) to solve the above
ODE,
\begin{equation*}
\begin{aligned}
    &\pt_t\l(e^{\lam t}\ll \nb g \rl^2_\hst\r) \leq Me^{\lam t} \l(e^{-2\lam t}\ll g(0)\rl^2_\hst\r) \leq Me^{-\lam t}\ll g(0)\rl^2_\hst,    \\
    &e^{\lam t}\ll \nb g \rl^2_\hst - \ll \nb g (0) \rl^2_\hst \leq M\ll g(0)\rl^2_\hst\frac{1}{\lam}(1-e^{-\lam t}),\\
    &\ll \nb g \rl^2_\hst  \leq e^{-\lam t}\l(\ll \nb g (0) \rl^2_\hst + \frac{M}{\lam}\ll g(0)\rl^2_\hst\r).
\end{aligned}
\end{equation*}
Therefore, by (\ref{eq: bd of nbhp 1}), one has
\begin{equation*}
    \ll \nb\hp \rl^2_\stt \leq \ll\frac{\hpinf}{\pinf} \rl_{L^\infty}\l(e^{-\lam t}\l(\ll \nb g (0)
    \rl^2_\hst + \frac{M}{\lam}\ll g(0)\rl^2_\hst\r) + 1\r).
\end{equation*}
\end{proof}

\end{document}